\makeatletter
\newif\ifLipics
\PassOptionsToPackage{noend,nosemicolon,vlined}{algorithm2e}
\ifLipics
    \documentclass[anonymous,USenglish,cleveref,notab]{socg-lipics-v2021}
    \titlerunning{Persistent Cycle Representatives and Barcode Functionals}
    \authorrunning{F.\ Lenzen, L.\ Renkin}
    \ccsdesc[500]{Mathematics of computing~Mathematical software}
    \ccsdesc[500]{Mathematics of computing~Geometric topology}
    \keywords{persistent homology, representative cycles, Alexander duality, merge trees}
    \let\c@author\relax
\else
    \documentclass[10pt,DIV=12]{scrartcl}
    \author{Fabian Lenzen, Leon Renkin}
    \date{today}
    \RedeclareSectionCommand[indent=0pt]{subparagraph}
    \usepackage{microtype,mathtools,amsthm,algorithm2e,cleveref}
    \usepackage{caption}
    \SetAlCapSty{sfbf}
    \SetAlCapFnt{\small}
    \SetAlCapNameFnt{\small}
    \newtheorem{theorem}{Theorem}
    \newtheorem{proposition}[theorem]{Proposition}
    \newtheorem{lemma}[theorem]{Theorem}
    \newtheorem{corollary}[theorem]{Corollary}
    \theoremstyle{definition}
    \newtheorem{definition}[theorem]{Definition}
    \theoremstyle{remark}
    \newtheorem{remark}[theorem]{Remark}
    \newtheorem{example}[theorem]{Example}
\fi
\title{Persistent Cycle Representatives and Generalized Landscapes for Codimension 1 Persistent Homology}
\author{Fabian Lenzen, Leon Renkin}
\date{\today}
\usepackage{xcolor,tikz,mathtools,amssymb,bbm,tikz-cd,csquotes}

\usepackage[alldates=year]{biblatex}
\DeclareSourcemap{
  \maps[datatype=bibtex]{
    \map[overwrite=true]{
      \step[fieldset=urldate, null]
      \step[fieldset=eventtitle, null]
      \step[fieldset=eventdate, null]
    }
    \map[overwrite=true]{
        \step[fieldsource = doi, final]
        \step[fieldset = url, null]
        \step[fieldset = isbn, null]
        \step[fieldset = issn, null]
    }
    \map[overwrite=true]{
        \step[fieldsource = issn, final]
        \step[fieldsource = isbn, final]
        \step[fieldset = url, null]
    }
  }
}
\usepackage{algorithm2e}
\renewcommand{\@algocf@capt@plain}{above}
\SetAlFnt{\small}

\SetKwSty{sfbf}
\usepackage[defaultlines=3,all]{nowidow}
\crefname{algocf}{algorithm}{algorithms}

\theoremstyle{remark}

\usepackage[colorinlistoftodos,textsize=tiny]{todonotes}

\newcommand{\Z}{\mathbb{Z}}
\newcommand{\N}{\mathbb{N}}
\newcommand{\R}{\mathbb{R}}

\newcommand{\OO}{\mathcal{O}}
\newcommand{\Merge}{\textsc{Merge}}
\newcommand{\Cancel}{\textsc{Cancel}}
\DeclareMathOperator{\im}{im}

\DeclareMathOperator{\sgn}{sgn}

\DeclarePairedDelimiter{\abs}{\lvert}{\rvert}

\newcommand{\one}{\mathbbm{1}}
\newcommand{\kk}{\mathbbm{k}}
\newcommand{\nthmax}[1]{\operatorname{\hbox{$#1$}\textsuperscript{th}-max}}
\newcommand{\wt}{\mathit{wt}}
\DeclareMathOperator{\Int}{Int}
\DeclareMathOperator{\Vol}{Vol}
\DeclareMathOperator{\EVol}{EVol}

\DeclareMathOperator{\supp}{supp}
\DeclarePairedDelimiterX{\Set}[1]{\{}{\}}{\setargs{#1}}
\NewDocumentCommand{\setargs}{>{\SplitArgument{1}{;}}m}{\setargsaux#1}
\NewDocumentCommand{\setargsaux}{mm}{\IfNoValueTF{#2}{#1} {#1\nonscript\:\delimsize\vert\allowbreak\nonscript\:\mathopen{}#2}}%
\DeclareFontFamily{OT1}{mathc}{}
\DeclareFontShape{OT1}{mathc}{m}{it}{<-> mathc10}{}
\DeclareMathAlphabet{\mathabxcal}{OT1}{mathc}{m}{it}
\newcommand{\barc}{\mathabxcal{Bar}}
\let\Sum\sum
\NewDocumentCommand{\SmashSum}{E{^_}{{}{}}}{\smashoperator{\Sum^{#1}_{#2}}}
\RenewDocumentCommand{\sum}{E{^_}{{}{}}}{\mathchoice{\SmashSum^{#1}_{#2}}{\Sum^{#1}_{#2}}{\Sum^{#1}_{#2}}{\Sum^{#1}_{#2}}}
\providecommand*{\boxast}{\mathbin{\mathpalette\@boxit{*}}}
\newcommand*{\@boxit}[2]{%
    \sbox0{$\m@th#1\Box$}%
    \ifx#1\displaystyle \ht0=\dimexpr\ht0+.05ex\relax \fi
    \ifx#1\textstyle \ht0=\dimexpr\ht0+.05ex\relax \fi
    \ifx#1\scriptstyle \ht0=\dimexpr\ht0+.04ex\relax \fi
    \ifx#1\scriptscriptstyle \ht0=\dimexpr\ht0+.065ex\relax \fi
    \sbox2{$#1\vcenter{}$}
    \rlap{%
        \hbox to \wd0{%
            \hfill
            \raisebox{%
                \dimexpr.5\dimexpr\ht0+\dp0\relax-\ht2\relax
            }{$\m@th#1#2$}%
            \hfill
        }%
    }%
    \Box
}

\AddToHook{env/lemma/begin}{\crefalias{theorem}{lemma}}
\AddToHook{env/proposition/begin}{\crefalias{theorem}{proposition}}
\AddToHook{env/corollary/begin}{\crefalias{theorem}{corollary}}
\AddToHook{env/definition/begin}{\crefalias{theorem}{definition}}
\AddToHook{env/example/begin}{\crefalias{theorem}{example}}
\AddToHook{env/condition/begin}{\crefalias{theorem}{condition}}
\AddToHook{env/assumption/begin}{\crefalias{theorem}{assumption}}

\newcommand{\Nobreak}{\makeatletter\nobreak\@afterheading\makeatother}

\makeatother
\begin{document}
\maketitle
\begin{abstract}
    For a filtered simplicial complex $K$ embedded in $\R^{d+1}$, the merge tree of the complement of $K$ induces a forest structure on the persistent homology $H_d(K)$ via Alexander duality.
    We prove that the connected components of $\R^{d+1}\setminus K_r$ correspond to representative cycles
    for a basis of $H_d(K_r)$ which are volume-optimal.
    By keeping track of how these representatives evolve with the filtration of $K$,
    we can equip each interval $I$ in the barcode of $H_d(K)$
    with a sequence of canonical representative cycles. 
    We develop and implement an efficient algorithm to compute the progression of cycles in time $\mathcal{O}((\#K)^2)$. 
    We apply functionals to these representatives,
    such as path length, enclosed volume, or total curvature.
    This way, we obtain a real-valued function for each interval, which captures geometric information about~$K$.
    Deriving from this construction, we introduce the \emph{generalized persistence landscapes}.
    Using the constant one-function as the functional, this construction gives back the standard persistence landscapes.
    Generalized landscapes can distinguish point clouds with similar persistent homology but distinct shape,
    which we demonstrate by concrete examples. 
\end{abstract}

\section{Introduction}
Persistent homology \cite{Robins:1999,EdelsbrunnerLetscherZomorodian:2002,Ghrist:2007,Carlsson:2009}, a central tool in topological data analysis,
seeks to estimate the homology of spaces from finite samples.
The general pipeline is that one assigns to a data set $X$ an $\R$-indexed filtration $K$ of simplicial complexes
and computes its persistent homology $H_k(K)$ w.r.t.\ a field $\kk$, which yields a so-called persistence module.
If $K_t$ is finite for all $t$,
a standard result \cite{Crawley-Boevey:2014a} asserts that the isomorphism type of $H_k(K)$ is uniquely determined by the \emph{barcode} $\barc(H_k(K))$ of $H_k(K)$,
which is the multiset of intervals in $\R$ corresponding to its indecomposable summands.
Intuitively, each interval~$I$ in $\barc(H_k(K))$ represents a distinct $k$-dimensional topological feature of $K$, 
which persists for all filtration values $t \in I$. 
An important vectorial summary of barcodes is the persistence landscape~\cite{Bubenik:2015}, 
which encodes the interval data of a barcode as a sequence of 1-Lipschitz functions, 
and thereby enables the application of classical statistical as well as machine learning tools.
Standard barcodes and persistence landscapes, however, summarize only the lifetimes of features and do not capture the geometry of the cycles that realize them. 

To incorporate such geometric information, one must consider the cycles representing these homology classes.
Since cycle representatives are not unique, one aims to find cycles that minimize or maximize a certain target function
among all cycles representing a given homology class.
While this problem is NP-hard in general, special cases are tractable in polynomial time.
One such case, on which we focus here,
is to find a \emph{volume-optimal cycle basis}  for an embedded complex $K$ in $\R^{d+1}$ in codimension one, i.e., for $H_d(K)$.
Instead of assigning a single representative cycle to each bar in the barcode,
we define \emph{cycle progression barcodes},
which contain a pair $(I,\gamma)$ for each $I \in \barc(H_d(K))$ with $\gamma\colon I\to Z_d(K_\infty)$,
such that the cycle $\gamma(t)$ optimally represents the summand of $H_d(K_t)$ corresponding to $I$.

By evaluating a functional $f\colon Z_d(K_\infty) \to \R$ on $\gamma$, 
we obtain a function $f\circ\gamma$ that contains information about the evolution of the shape of $\gamma$. For example, $(f\circ\gamma)(t)$ could be the arc length of $\gamma(t)$, or the volume enclosed by it. 
From the convolution of $f\circ\gamma$ with the indicator function of $I$, we construct generalized persistence landscapes,
which are given by a sequence of real-valued functions and specialize to standard persistence landscapes for the constant one-functional.

\subsection{Related work}
Computing optimal homology representatives has widely been studied in the literature \cite{EricksonWhittlesey:2005,DeySunWang:2009,DeyHiraniKrishnamoorthy:2010,EscolarHiraoka:2016}.
For a simplicial complex $L$ and $\kk = \Z/2\Z$, 
we extend a weight function $\wt\colon L^{(k)} \to \R_{\geq 0}$ linearly to $C_k(L)$.
A cycle $z \in Z_k(L)$ \emph{optimally represents} the homology class $[z]$
if it minimizes $\wt$ among all elements of $[z]$.
A set $S \subseteq Z_k(L)$ is an \emph{optimal cycle basis} if $\Set{[z]; z \in S}$ is a basis of $H_k(L)$
that minimizes $\sum_{z\in S} \wt(z)$.
Both problems can be phrased as integer programming problems \cite{DeyHiraniKrishnamoorthy:2010,EscolarHiraoka:2016}, which in general is NP-hard.
Polynomial time algorithms exist in special cases,
such as $k=1$ \cite{EricksonWhittlesey:2005,DeySunWang:2009},
or if $H_k(L',L'')$ is torsion-free for all pure subcomplexes $L'' \subseteq L' \subseteq L$
of dimension $k$ and $k+1$ \cite{DeyHiraniKrishnamoorthy:2010}. 

A different kind of optimality for $k$-cycles can be considered if a weight function $\wt\colon L^{(k+1)} \to \R_{\geq 0}$ is considered;
in this case, one speaks of \emph{volume-optimal representatives} \cite{Obayashi:2018}.
If $\dim L = d+1$, and $\abs{L} \subseteq \R^{d+1}$, then Alexander duality gives a correspondence
between volume-optimal $d$-cycles and connected components of $\R^{d+1}\setminus\abs{L}$.
This correspondence is compatible with filtrations,
and a polynomial time algorithm for computing volume-optimal cycles (w.r.t. the constant weight function $\sigma\mapsto1$) in codimension one
has been presented in \cite{Schweinhart:2016,Obayashi:2018}.
It is implemented in the Homcloud package \cite{ObayashiNakamuraEtAl:2022a}.

In contrast to \cite{Schweinhart:2016,Obayashi:2018,ObayashiNakamuraEtAl:2022a}, which assign a single optimal representative to each bar,
our approach provides a sequence of optimal representatives for each bar, updating exactly at filtration events. 
An algorithm for tracking representative cycles across a sequence of simplicial complexes was also presented in \cite{GambleChintakuntaKrim:2014}, 
which uses heuristics such as shortest cycles to choose representatives
and does not enforce geometric optimality.
In contrast, we compute canoncial and volume-optimal representative cycles that evolve across the filtration via the persistence forest.

\subsection{Contribution}
We build on \cite{Obayashi:2018}, giving a new proof for the volume-optimally (w.r.t. any weight function) of the codimension-one cycles obtained via Alexander duality.
Based on this, we introduce the \emph{persistence forest}
of a $(d+1)$-dimensional filtered complex $K \subseteq \R^{d+1}$;
which is closely related to persistence tree as 
defined and computed in \cite{Schweinhart:2016,Obayashi:2018}, see Remark~\ref{rmk:persistence-tree}.
We extend the Alexander duality correspondence and persistence forests to the setting of \emph{signed cycles}, which contain strictly more information than ordinary cycles.
We provide a polynomial time algorithm that computes the (signed) persistence forest of $K$,
using an approach similar to the algorithm presented in \cite{Obayashi:2018}.

From the persistence forest, we derive the construction of \emph{cycle progression barcodes},
which enrich each bar in $\barc(H_d(K))$ by a sequence of volume-optimal representative cycles.
We show that if $K$ is a contractible $(d+1)$-dimensional filtered simplicial complex embedded in $\R^{d+1}$,
one can compute the cycle progression barcode that minimizes the enclosed weight
for any linear weight function $\wt\colon K^{(d+1)}_\infty \to \R_{\geq0}$ in time $\OO((\#K)^2)$. 
Using progression barcodes, we define \emph{generalized landscapes},
which take geometric information of the represenative cycles into account.
The standard persistence landscapes~\cite{Bubenik:2015} arise as special case of the generalized landscapes.
\Cref{fig:intro-example} shows an example of such a generalized persistence landscape, and the cycle progression barcode giving rise to it.
We note that our concept of generalized landscapes
differs from the one introduced in \cite{BerryChenCisewski-KeheFasy:2018}:
while the latter generalizes landscapes derived from barcodes
by admitting other kernel functions than the distance to the closer boundary $t \mapsto \max\{0, \min\{t-b,d-t\}\}$ for an interval $I = [b,d)$,
we use kernels that take geometric information about the cycle progression representing $I$ into account.

We provide algorithms and implementations for efficiently computing persistence forests, cycle progression barcodes and generalized landscapes. Computing these structures for a Delaunay triangulation of 100\,000 points in $\R^2$ can be done within seconds.
We demonstrate the effectiveness of generalized landscapes as shape descriptors by distinguishing point clouds with similar barcodes but distinct shape. Additionally, we show that using signed chains enables us to detect thin bridges between clusters more accurately.

\begin{figure}
    \centering
    \raisebox{-.5\height}{\includegraphics{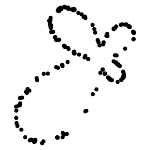}}
    \raisebox{.5\depth-.5\height}{
        \begin{tabular}{c}
            \includegraphics{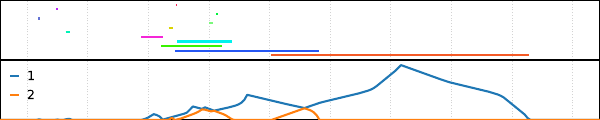} \\
            \includegraphics{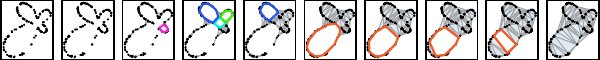}
        \end{tabular}
    }
    \caption{%
        Top right: The $H_1$-barcode of the $\alpha$-complex of the point cloud on the left,
        and the first two associated generalized persistence landscapes w.r.t.\ $\frac{(\text{arc length})^2}{\text{enclosed area}} - 4\pi$.
        Bottom right: The cycle progressions corresponding to the bars of the same color.
        Each picture corresponds to the state of the cycles at the time indicated by a dashed vertical lines.
    }
    \label{fig:intro-example}
\end{figure}
\subsection{Structure}
We briefly review the relevant concepts from the theory of persistent homology in \cref{sec:background}.
We discuss Alexander duality and volume-optimal representative cycles in \cref{sec:optimal-cycles}.
In \cref{sec:forests}, we introduce signed chains and cycles, and present (signed or unsigned) persistence forests including an algorithm to compute them.
In \cref{sec:landscapes}, we present the definition and computation of cycle progression barcodes and generalized landscapes.
We show the effectiveness of signed cycles and generalized landscapes as shape descriptors in \cref{sec:examples-and-benchmarks}.

\section{Preliminaries}
\label{sec:background}
A \emph{filtered simplicial \emph{(resp.} cellular\emph{)} complex} is an assignment $K$ of a finite simplicial (resp. cellular) complex $K_t$ for every $t \in \R$,
such that $K_s \subseteq K_t$ for each $s \leq t$.
We often assume that $K$ is a $(d+1)$-dimensional complex that can be extended to a triangulation of a $(d+1)$-sphere $S^{d+1}$.
If $K_\infty$ is embedded in $\R^{d+1}$ and triangulates the convex hull of its vertices,
then $\partial K_\infty$ is a $d$-sphere. We can embed $K_\infty$ into $S^{d+1}$ by using the one-point compactification of $\R^{d+1}$.
We denote the point in $S^{d+1} \setminus \R^{d+1}$ by $\infty$.
A triangulation $T$ of $S^{d+1}$ is given by $T = K_\infty \cup (\partial K_\infty * \infty)$, where $*$ denotes the simplicial join.
For example, the $\alpha$-filtration \cite{EdelsbrunnerKirkpatrickEtAl:1983} 
or the Čech--Delaunay filtration \cite{BauerEdelsbrunner:2017} of a point cloud $X \subseteq \R^{d+1}$
are of this kind.

\smallskip
Let $\kk$ be an arbitrary field.
A \emph{persistence module} $M:(\R,\leq)\to\textbf{vec}_\kk$ is a functor from the poset category $(\R,\leq)$ into the category of $\kk$-vector spaces $\textbf{vec}_\kk$.
We call $M$ \emph{pointwise finite dimensional} if $M_t$ is a finite dimensional vector space for each $t$.
For an interval $I \subseteq \R$, we write $\Int(I)$ for the interval module supported on $I$.
If $M$ is pointwise finite dimensional, then there is an isomorphism
$\bigoplus_{I \in \barc(M)} \Int(I) \to M$
for a uniquely determined multiset $\barc(M)$ of intervals, called the \emph{barcode} of $M$ \cite{Crawley-Boevey:2014a}.

Let $K$ be a filtered simplicial (or cellular) complex as above.
Taking $k$-th homology of $K$ (with coefficients in $\kk$) defines a persistence module $H_k(K)$,
where $H_k(K)_t = H_k(K_t)$, and $H_k(K)_s\to H_k(K)_t$ is the map induced by the inclusion $K_s \hookrightarrow K_t$.
Since $K_t$ is finite for all $t$, $H_k(K)$ has a barcode.
Analogously, taking $k$-th cohomology defines a pointwise finite dimensional persistence module $H^k(K)$,
with $H^k(K)_t = H^k(K_{-t})$.
The reduced (co)homology modules $\tilde{H}_k(K)$ and $\tilde{H}^k(K)$ are defined analogously.

\section{Volume-optimal cycles in codimension-one}
\label{sec:optimal-cycles}
\subsection{Alexander Duality}
We briefly revisit Alexander duality and refer to \cites[\S17.4, \S17.11]{FomenkoFuchs:2016}[Appendix~A.2]{Obayashi:2018}{BjornerTancer:2009} for proofs and details.
\begin{figure}
    \centering
    \includegraphics[height=2cm]{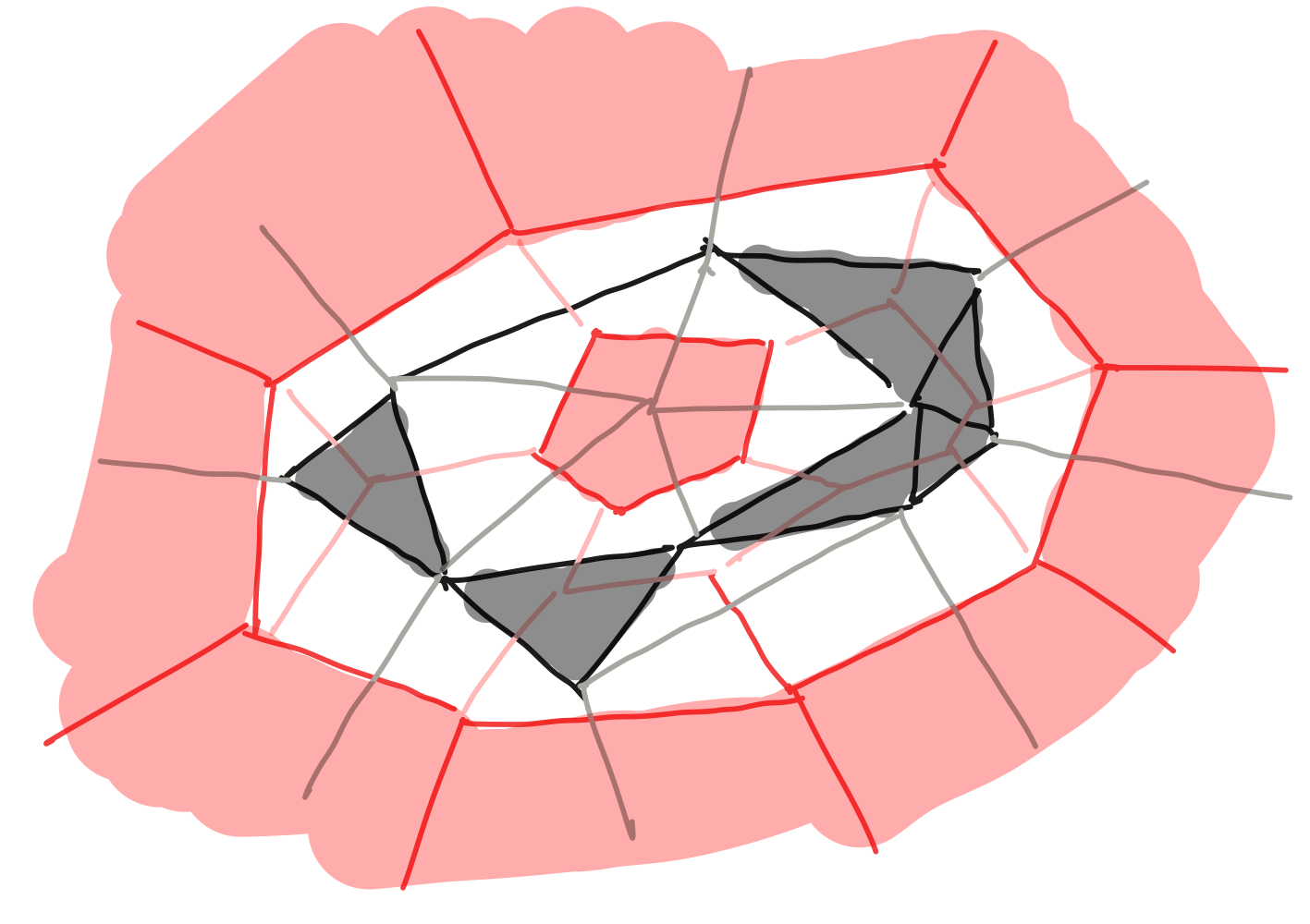}
    \caption{The triangulation $T$ of $S^2$ (gray), its subcomplex $L \subseteq T$ (black),
        the cell complex decomposition $\bar{T}$ of $S^2$ dual to $T$ (pink),
        and its subcomplex $\bar{X} \subseteq \bar{T}$ dual to $X \subseteq T$ (red).}
    \label{fig:alexander-duality}
\end{figure}
Let $T$ be a triangulation of the $(d+1)$-sphere $S^{d+1}$ and $L \subsetneq T$ be a subcomplex.
There exists a polyhedral decomposition $\bar{T}$ of $S^{d+1}$, called the \emph{dual decomposition} of $T$,
whose $k$-cells $\bar{T}^{(k)}$ are in bijection with $T^{(d+1-k)}$.
The cell corresponding to a simplex $\sigma \in T$ is denoted by $\bar{\sigma}$.
After choosing appropriate orientations for the cells of~$\bar{T}$,
the matrix representing the boundary operator of $\bar{T}$ with respect to the standard basis
is the transpose of that of $T$.
If $L \subseteq T$ is a simplicial subcomplex, then $\bar{L} \coloneqq \Set{\bar{\sigma}; \sigma \in T \setminus L}$
is a cellular subcomplex of $\bar{T}$; see \cref{fig:alexander-duality}.

\begin{theorem}[Alexander duality]
    \label{Alexander Duality}
    For every $k\in \N$, there exists a natural isomorphism ${\tilde{H}^k(\bar{L}) \cong \tilde{H}_{d-k}(L)}$ in reduced (co)homology.
\end{theorem}

We will make the isomorphism explicit for $k = 0$.
First, to choose an \enquote{appropriate orientation} of the $1$-cells in $\bar{T}$,
let $\omega \in Z^{d+1}(T, \Z)$ such that $[\omega]$ generates $H^{d+1}(T, \Z)$.
Then for each $\tau \in T^{(d+1)}$ with face $\sigma$, we orient the dual cell $\bar{\sigma}$
such that $\bar{\tau}^*(\partial\bar{\sigma}) = \sigma^*(\partial\tau) \omega(\tau)$,
where $\bar{\tau}^*$ and $\sigma^*$ denote the basis elements of $Z^0(\bar{T})$ and $Z^d(T)$ dual to $\bar{\tau}$ and $\sigma$, respectively.
Now, for a complex $L \subseteq T$, consider the map
\begin{equation}
    \label{eq:def-Phi}
    \begin{split}
        \Phi_L\colon C^0(\bar{L}) &\to C_{d+1}(T),\\
        \bar{\sigma}^* &\mapsto \omega(\sigma)\sigma.
    \end{split}
\end{equation}
One can show that for $\gamma \in Z^0(\bar{L})$, one has $\partial\Phi_L(\gamma) \in Z_d(L)$.
The isomorphism from \cref{Alexander Duality} now is given by
\begin{equation}
    \label{eq:alexander-duality}
    \begin{split}
    \tilde{\Phi}_L\colon \tilde{H}^0(\bar{L}) &\to H_d(L)\\
    [\gamma] &\mapsto [\partial\Phi_L(\gamma)].
    \end{split}
\end{equation}

\subsection{Optimal cycles}
The goal of this section is to show that the $d$-cycles on $L$ corresponding to the connected components of $\bar{L}$ under Alexander duality
give \enquote{nice} representatives for a basis of $H_d(L)$.
Assume that $L \subsetneq T$; then there exists a $(d+1)$-cell $\infty \in T\setminus L$.
Its dual $\bar\infty$ is a vertex of~$\bar{L}$.
Let $T^* \coloneqq T \setminus \{\infty\}$.
Then $T^*$ is a contractible, which immediately yields the following:

\begin{lemma}
    \label{thm:interior}
    For every $z \in Z_d(T^*)$ there exists a uniquely determined chain $\Int z \in C_{d+1}(T^*)$ bounded by $z$.
\end{lemma}

Define the linear map
\begin{equation}
    \begin{split}
        \Psi_L\colon Z_d(L) &\to C^0(\bar{L}),\\
        c &\mapsto \Phi^{-1}_L(\Int c).
    \end{split}
\end{equation}
Unravelling this definition immediately yields the following:

\begin{lemma}
    \label{thm:R1}
    The map $R \coloneqq \partial\Phi_L\Psi_L\colon Z_d(L) \to Z_d(L)$ satisfies
    \(
        R(z) = \sum_{\sigma \in T^* \setminus L} \sigma^*(z)\sigma
    \).
\end{lemma}

Intuitively, $R(z)$ strips away the $B_d(L)$-part of $z$.
The following makes this precise.

\begin{lemma}
    \label{thm:R2}
    For all $z \in Z_d(L)$, we have $z - R(z) \in B_d(L)$.
\end{lemma}
\begin{proof}
    Expressing $z$ in the standard basis of $C_d(T^*)$ yields
    \[
        z = \sum_{\sigma \in L} \sigma^*(z)\sigma  + \sum_{\sigma \in T^*\setminus L} \sigma^*(z)\sigma.
    \]
    By \cref{thm:R1}, the second sum equals $R(z)$, which yields the claim.
\end{proof}

\noindent The \emph{support} of a cycle $z \in Z_k(T)$ and a cochain $\gamma \in C^k(T)$ is
\begin{align}
    \supp z &\coloneqq\Set{\sigma\in T^{(k)}; \sigma^*(z)\ne 0},&
    \supp \gamma &\coloneqq \Set{\sigma\in T^{(k)}; \gamma(\sigma) \neq 0}.\\
\intertext{%
    Let $\wt\colon (T^*)^{(d+1)}\to\R_{\geq 0}$ be a weight function.
    We define the weight of $z$ and $\gamma$ by
}
    \wt(z) &\coloneqq \sum_{\sigma\in\supp (\Int z)}\wt(\sigma),&
    \wt(\gamma) &\coloneqq \sum_{\sigma\in\supp\gamma} \wt(\sigma).
\end{align}

\begin{definition}
    A \emph{cycle basis} of $L$ is a set $S \subseteq Z_d(L)$ such that $\Set{[z]; z \in S}$ is a basis of $\tilde{H}_d(L)$.
    A \emph{cocycle basis} of $\bar{L}$ is a set $S \subseteq Z^0(\bar{L})$ such that $\Set{[\gamma]; \gamma \in S}$ is a basis of $\tilde{H}^0(\bar{L})$.
    A (co)cycle basis $S$ is called \emph{minimal} if it minimizes $\wt(S)\coloneqq\sum_{z \in S} \wt(z)$.
\end{definition}

Per Alexander duality, the map $\partial\Phi_L$ sends cocycle bases of $\tilde{H}_0(\bar{L})$ to cycle basis of $\tilde{H}_d(L)$.

\begin{lemma}
    The image of $\partial\Phi_L$ contains a minimal cycle basis.
    If $\wt(\sigma) > 0$ for all $\sigma \in L^{(d+1)}$, then every minimal cycle basis of $L$ lies in $\im\partial\Phi_L$. 
\end{lemma}
\begin{proof}
    Let $S \subseteq Z_d(L)$ be a cycle basis.
    By \cref{thm:R2}, $R(S) \coloneqq \Set{R(z); z\in S}$ is also a cycle basis,
    which is contained in $\im\partial\Phi_L$.
    Let $z \in S$ be a cycle.
    Then $\supp R(z) \subseteq (T^*)^{(d)} \setminus L^{(d)}$ and $\supp(z - R(z)) \subseteq L^{(d)}$ are disjoint, 
    which implies $\wt(z) = \wt(R(z)) + \wt(z-R(z))$.
    Since the second term is nonnegative, we get $\wt(z) \geq \wt(R(z))$, so in particular
     $\wt(S)\geq\wt(R(S))$.
    If $\wt(\sigma) > 0$ for all $\sigma\in L^{(d+1)}$, then $\wt(z) = \wt(R(z))$
    precisely if $z = R(z)$.
\end{proof}

\begin{remark}
\label{sec:pi0-as-cocycle}    
We view the set $\pi_0(L)$ of connected components of $L$ as a subset of $Z^0(\bar{L})$
by identifying a connected component $c$ with the cocycle $\sum_{v \in c} v^*$ supported on it.
In this sense, $\pi_0(\bar{L})$ is a basis of the vector space $Z^0(\bar{L})$.
Let $\tilde{\pi}_0(\bar{L})$ be the subset of $\pi_0(\bar{L})$ not containing the the connected component of $\bar{\infty}$.
Then $\tilde{\pi}_0(\bar{L})$ is a cocycle basis for $\tilde{H}^0(\bar{L})$.
\end{remark}

\begin{theorem} \label{thm:weigth-optimality}
        The set $\Set{\partial \Phi_L(c); c \in \tilde{\pi}_0(\bar{L}))}$ is a minimal cycle basis of $L$.
        If $\wt(\sigma) > 0$ for all $\sigma$, it is unique up to rescaling.
\end{theorem}
\begin{proof}
    A cycle basis $S \subseteq Z_d(L)$ is minimal if and only if $S = \partial\Phi_L(S')$
    for a minimal cocycle basis $S' \subseteq Z^0(\bar{L})$.
    Therefore, it suffices to show that $\tilde{\pi}_0(\bar{L})$ is a minimal cocycle basis,
    and it is the only one if $\wt(\sigma) > 0$ for all $\sigma$.
    The elements of $\tilde{\pi}_0(\bar{L})$ have pairwise disjoint support,
    so any cocycle $\gamma = \sum_{c\in\tilde{\pi}_0(\bar{L})} \mu_{\gamma c}c$ with scalar coefficients $\mu_{\gamma c}$ has weight
    \(
        \wt(\gamma) = \sum_{\mu_{\gamma c \neq 0}} \wt(c).
    \)
    Let $S' \subseteq Z^0(\bar{L})$ be any cocycle basis.
    We obtain that
    \[
        \sum_{\gamma\in S'} \wt(\gamma) = \sum_{\substack{\gamma\in S', c\in \tilde{\pi}_0(\bar{L})\\\mu_{\gamma c} \neq 0}} \wt(c).
    \]
    Since the $\mu_{\gamma c}$ encode a change of basis, at least one $\mu_{\gamma c}$ is non zero for each $c\in\tilde{\pi}_0(\bar{L})$.
    This implies that
    \(
        \sum_{\gamma\in S'} \wt(\gamma) \geq \sum_{c\in\tilde{\pi}_0(\bar{L})} \wt(c).
    \)
    If $\wt(\sigma) > 0$ for all $\sigma$, then equality holds if and only if precisely one $\mu_{\gamma c}$ is non-zero for each $c$,
    which implies that each $\gamma \in S'$ is a scalar multiple of some $c \in \tilde{\pi}_0(\bar{L})$.
\end{proof}

\subsection{Filtrations}

\begin{figure}
    \centering
    \includegraphics{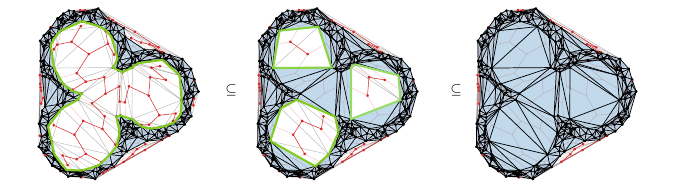}
    \caption{%
        The bounded connected components of the dual filtered complex $\bar K$ (red)
        correspond to a weight-minimal basis (green) for the codimension-one homology group of $K$ (black and blue). 
    }
    \label{fig:filtration-dual}
\end{figure}

The isomorphism $\tilde{\Phi}_L$ from \eqref{eq:alexander-duality} is natural in the sense that for any subcomplex $L'\subseteq L$, the inclusion $\iota\colon L' \hookrightarrow L$ 
and the induced inclusion $\bar{\iota}\colon \bar{L} \hookrightarrow \bar{L'}$ yield the commutative diagram
\[
    \begin{tikzcd}
        \tilde{H}^0(\bar{L'}))\arrow[d, "\bar{\iota}^*"', hook] \arrow[r, "\tilde{\Phi}_{L'}"]  &  H_{d}(L') \arrow[d, hook, "\iota_*"] \\
        \tilde{H}^0(\bar{L})) \arrow[r,"\tilde{\Phi}_L"'] &  H_{d}(L).
    \end{tikzcd}
\]

\Cref{Alexander Duality} immediately yields the following persistent version of Alexander duality.
For a filtered simplical complex $K$ with $K_\infty \subseteq T^*$,
denote by $\bar{K} \subseteq \bar{T}$ the filtered cell complex given by $\bar{K}_t \coloneqq \overline{K_{-t}}$.
Then we have:

\begin{corollary}\label{cor:filtration-alexander-duality}
    The maps $\tilde{\Phi}_{K_t}\colon \tilde{H}^0(\bar{K}_t)\to \tilde{H}_{d}(K_t)$ from \cref{Alexander Duality}
    with $t \in \R$
    induce an isomorphism of persistence modules $\tilde{\Phi}_K\colon \tilde{H}^0(\bar{K})\to \tilde{H}_{d}(K)$;
    see \cref{fig:filtration-dual}.
\end{corollary}


\subparagraph*{Orientations}
Recall that the construction of $\tilde{\Phi}_K$
depends on the choice of a cochain $\omega \in Z^{d+1}(T, \Z)$ such that $[\omega]$ generates $H^{d+1}(T, \Z)$.
We call an element $\omega \in Z^{d+1}(K_\infty, \Z)$ an \emph{orientation} of $K$ if  $[\omega]$ extends to a generator $H^{d+1}(T,\Z)$. 
An orientation $\omega$ for $K$ is unique up to sign and satisfies $\omega(\sigma)\in\{\pm1\}$ for every $\sigma\in K_\infty^{(d+1)}$.

\begin{example}
    \label{ex:orientation}
    If $K_\infty$ is embedded in $\R^{d+1}$,
    then an orientation of $K$ is given by the cocycle $\omega\colon [v_0,\dotsc,v_{d+1}] \mapsto \sgn\det(v_1-v_0,\dotsc,v_{d+1}-v_0)$.
\end{example}

\section{Merge trees and persistence forests}
\label{sec:forests}
We say that a filtered simplicial complex $K$ is a \emph{closed} (resp.\ \emph{open}) \emph{sublevel filtration}
if there is a poset map $\tau\colon K_\infty \to \R\cup\{-\infty\}$
such that $K_t = \tau^{-1}([-\infty, t])$ (resp.\ $K_t = \tau^{-1}([-\infty, t))$).
We call $\tau$ the filtration function of $K$.
The \emph{Hasse diagram} of a poset $P$ is the directed graph $(P, E)$ with an edge $p \to q$ if $p < q$ and there exists no $r$ such that $p < r < q$.

\begin{definition}
    \label{def:merge-tree}
    Let $K$ be a finite closed (resp. open) sublevel filtration.
    Consider the set $S = \bigcup_{t \in \R} \pi_0(K_t) \times \{t\}$
    (resp.\ $S = \bigcup_{t \in \R} \pi_0(\bigcap_{s > t} K_s) \times \{t\}$),
    partially ordered by $(c_1, t_1) \leq (c_2, t_2)$ if $c_1 \subseteq c_2$ and $t_1 \leq t_2$.
    The \emph{(labeled) merge tree} of $K$ is the Hasse diagram $M(K)$ of its finite subposet
    \[
        \Set[\big]{(c, t) \in S; \textstyle\bigcup_{(c', t') < (c, t)} c' \subsetneq c}.
    \]
    If $K_{-\infty} \coloneqq \bigcap_{t \in \R} K_t$ is connected, 
    then the \emph{reduced (labeled) merge forest} $\tilde{M}(K)$ of $K$ 
    is obtained from $M(K)$ as follows.
    Let $\Gamma$ be the maximal directed path emenating from the leaf $-\infty$ corresponding to the unique class in $\pi_0({K}_{-\infty})$,
    and let $U$ be the set of vertices on that path.
    One then obtains $\tilde{M}(K)$ from $M(K)$ by
        replacing every edge $(c', t') \to (c,t)$ for $(c,t) \in U$ by an edge $(c',t'') \to (c', t)$ for a new vertex $(c',t)$, and
        deleting all edges in $\Gamma$ and vertices in $U$, including the leaf $-\infty$;
    see \cref{fig:reduced-merge-tree}.
    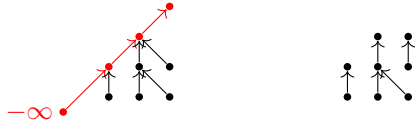
\begin{figure}
        \centering
        \usetikzlibrary{positioning}
    	\tikzset{every picture/.append style={x=4mm, y=4mm, node distance=1 and 1}}
    	\begin{tikzpicture}[baseline=11]
    		\path[nodes={shape=circle, fill, inner sep=1pt}, every edge/.append style={->}]
    		(.5,.5) node[red, label={[red]left:$-\infty$}] (0) {}
    		(2,2) node[red] (1) {} (0)   edge[red] (1)
    		(2,1) node (11) {} (11)  edge (1)
    		(3,3) node[red] (2)  {} (1)   edge[red] (2)
    		(3,2) node (211) {} (211) edge (2)
    		(3,1) node (212) {} (212) edge (211)
    		(4,1) node (213) {} (213) edge (211)
    		(4,2) node (22)  {} (22)   edge (2)
    		(4,4) node[red] (3) {} (2) edge[red] (3);
    	\end{tikzpicture}
        \hspace{4em}
    	\begin{tikzpicture}[baseline=11]
    		\path[nodes={shape=circle, fill, inner sep=1pt}, every edge/.append style={->}]
    		(.5,.5) node[fill=none] (0) {}
    		(2,2) node[fill] (1) {}
    		(2,1) node (11) {} (11)  edge (1)
    		(3,3) node (2)  {}
    		(3,2) node (211) {} (211) edge (2)
    		(3,1) node (212) {} (212) edge (211)
    		(4,1) node (213) {} (213) edge (211)
    		(4,2) node (22)  {} 
    		(4,3) node (2a)  {} (22)   edge (2a)
    		node[fill=none] (3) {};
    	\end{tikzpicture}
        \caption{Passing from the merge tree $M(K)$ (left) to the reduced merge forest $\tilde{M}(K)$ (right).}
        \label{fig:reduced-merge-tree}
    \end{figure}
\end{definition}

Let $T$ be a triangulation of $S^{d+1}$ and $K \subsetneq T$ be a closed sublevel set filtration
of a $(d+1)$-dimensional complex with filtration function $\tau$ such that $K_\infty$ is contractible. 
Then the dual filtration $\bar{K}$ is an open sublevel set filtration with filtration function $-\tau$.
The intersection $\bar{K}_{-\infty} \coloneqq \bigcap_{t \in \R} \bar{K}_t$
has a single connected component containing the vertex $\bar{\infty}$ of $\bar{K}_\infty$.
Recall from \eqref{sec:pi0-as-cocycle} that for a complex $L$, we identify $\tilde{\pi}_0(L)$ with a subset of $Z^0(L)$.
Consider the map $\Theta_L \coloneqq \partial\Phi_L\colon Z^0(L) \to Z_d(L)$,
where $\Phi_L$ denotes the map from \eqref{eq:def-Phi}.
Note that the definition of $\Theta_L$ depends on the choice of an orientation $\omega$ of $K$.

\begin{definition}
    \label{def:persistence-forest}
    Let $K$ be as above.
    Let $(\bar{V}, \bar{E}) = \tilde{M}(\bar{K})$ be the reduced merge forest of $\bar{K}$.
    The \emph{persistence forest} $T(K)$ of $K$ (w.r.t.\ $\omega$) is the directed graph $(V, E)$
    with $V \coloneqq \Set{(\Theta_{K_t}(c), -t); (c, t) \in \bar{V}}$
    and $E$ obtained from $\bar{E}$ in the obvious way.
\end{definition}

Because the orientation $\omega$ on $K$ is unique up to sign, so is the persistence forest.
We thus often omit the orientation, and speak of \enquote{the} persistence forest of $K$.

\begin{remark}
    \label{rmk:persistence-tree}
    The \emph{persistence tree} from \cite{Obayashi:2018} is related to our persistence forest as follows.
    Let $(V, E)$ be the persistence forest of a simplex-wise filtration $K$, 
    and $(V', E')$ be a persistence tree in the sense of \cite{Obayashi:2018}.
    Then $V' = \{\infty\} \cup \Set{t; \text{$V$ has a leaf $(z,t)$}}$,
    and $E'$ contains a directed edge $d \stackrel{b}\to s$
    for each two leaves $(z_d, d)$ and $(z_s, s)$ with a common ancestor $(z_b, b)$ 
    such that $d < s$ and both are the leaves with maximal $d$ and $s$
    in their respective subtrees below $(z_b, b)$.
    Additionally, $E'$ contains an edge $d \stackrel{b}\to \infty$ for every root $(z_b,b)$ in $V$,
    where $(z_d, d)$ is the leaf of $(z_b, b)$ with maximal $d$.
    See \cref{fig:persistence-tree-and-forest} for an example comparing the two constructions.
    \begin{figure}
        \centering
        \begin{tikzpicture}[every edge quotes/.append style={inner sep=0.5pt, font=\scriptsize}]
    		\path[node distance=1] 
            { [nodes={inner sep=1pt, fill, shape=circle}]
    			node (1) {}
    			node[right=of 1] (2) {}
    			node[above right=sin(60) and cos(60) of 1] (3) {}
    			node[above right=sin(60) and cos(60) of 2] (4) {}
    			node[above right=sin(60) and cos(60) of 3] (5) {}
    		}
    		(1) edge["1"'] (2)
    		(1) edge["2"] (3)
    		(2) edge["3"'] (4)
    		(3) edge["4"] (4)
    		(3) edge["5"] (5)
    		(4) edge["6"'] (5)
    		(2) edge["7"] (3)
    		{ [nodes={red, font=\scriptsize}]
    			node at (barycentric cs:2=1,3=1,4=1) {8}
    			node at (barycentric cs:1=1,2=1,3=1) {9}
    			node at (barycentric cs:3=1,4=1,5=1) {10}
    		};		
    	\end{tikzpicture}
        \qquad\qquad
        \begin{tikzpicture}[x=.5cm, y=.5cm, every edge/.append style={->}]
            \path
            node (8) {8}
            node[right=1 of 8] (9) {9}
            node[below right=.25 and 1 of 9] (inf) {$\infty$}
            node[below left=.25 and 1 of inf] (10) {10}
            (8) edge["7"] (9)
            (9) edge["4"] (inf)
            (10) edge["6"'] (inf);
        \end{tikzpicture}
        \qquad\qquad
        \begin{tikzpicture}[every label/.append style={shape=rectangle, font=\scriptsize, inner sep=.5pt}, y=.25cm, y={(0,-1)}]
        	\path
        	(0,2.5) edge[->,gray!25] (0,10.5)
        	foreach \i in {3,...,10} {
        		(0,\i) node[label={[gray]left:$\i$}] {} edge[gray!25] (4,\i)
        	}
        	{ [nodes={fill,shape=circle,inner sep=1pt}] 
        		node[label=below:$\partial(10)$] (10) at (1,10) {}
        		node[label=above:$\partial(10)$] (6) at (1,6) {}
        		node[label=below:$\partial(9)$] (9) at (2,9) {}
        		node[label=below:$\partial(8)$] (8) at (3,8) {}
        		node[label=above right:$\partial(8+9)$] (7) at (2.5,7) {}
        		node[label=above:$\partial(8+9)$] (4) at (2.5,4) {}
        	}
            {[every edge/.append style={->}]
            	(10) edge[blue] (6)
            	(9) edge[green, out=90, in=180] (7)
            	(8) edge[red, out=90, in=0] (7)
            	(7) edge[green] (4)
            };
        \end{tikzpicture}
        \caption{%
            Left: a filtered simplicial complex.
            The filtration values are given in black for the edges
            and in red for the triangles.
            The filtration values of the vertices are not relevant and thus omitted.
            Middle: The persistence tree of this complex as constructed in \cite{Schweinhart:2016,Obayashi:2018}.
            Right: The persistence forest of this complex as defined in \cref{def:persistence-forest}.
            The colors correspond to \cref{fig:cycle-progression}.
        }
        \label{fig:persistence-tree-and-forest}
    \end{figure}
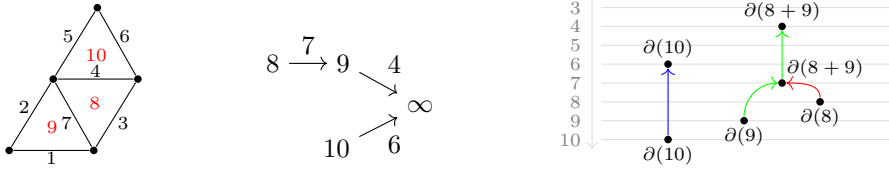
\end{remark}

\begin{lemma}\label{eq:interior-of-unsigned-dual-boundary}
    Let $K$ be as in \cref{def:persistence-forest}. For any $c\in \pi_0(\bar{K}_t)$, we have 
    \[
        \supp( \Int(\Theta_{K_t}(c))) = \Set{\sigma\in K^{(d+1)}_\infty; \bar\sigma\in c}.
    \]
\end{lemma}
\begin{proof}
    Let $\omega$ be the orientation of $K$. We obtain
    \[
        \Int\Theta_{K_t}(c)
        =\Int \partial\Phi_{K_t}(c)
        =\Phi_{K_t}(c) 
        = \Phi_{K_t} \Bigl(\sum_{\bar\sigma \in c}\bar\sigma^*\Bigr)
        =\sum_{\bar\sigma \in c}\omega(\sigma)\sigma. \qedhere
    \]
\end{proof}

\begin{lemma}\label{eq:dual-boundary-compatibility}
    For $s\leq t$, and component $c\in \pi_0(\overline{K_s})$,
    let $C \coloneqq \Set{c' \in \pi_0(\overline{K_t}); c'\subseteq c}$.
    Then the following equations holds in $H_d(K_t)$:
    \[
        [\Theta_{K_s}(c)]=\sum_{c' \in C}[\Theta_{K_t}(c')].
    \]
\end{lemma}
\begin{proof}
    First, we note that $\Theta_{K_s}(c)\in Z_d(K_s)\subseteq Z_d(K_t)$.
    With $c'' = c\setminus \bigcup C$, we get
    \[
        \Theta_{K_s}(c) = \sum_{\bar\sigma \in c}\omega(\sigma)\partial\sigma 
        = \sum_{\bar\sigma \in c''} \omega(\sigma)\partial\sigma + \sum_{c' \in C}\sum_{\bar\sigma \in c'} \omega(\sigma)\partial\sigma 
        = \partial\Bigl(\sum_{\bar\sigma \in c''} \omega(\sigma)\sigma\Bigr) +   \sum_{c' \in C} \Theta_{K_t}(c').
    \]
    The result now follows from that fact that $\sum_{\bar\sigma \in c''} \omega(\sigma)\sigma\in C_d(K_t)$ because $c''\cap \bar K_t=\emptyset$.
\end{proof}

\subsection{Signed chains and cycles}
\label{sec:signed-chains}
Observe that the complexes $X$ and $Y$ in \cref{fig:signed-vs-unsigned} have the same set of $1$-cycles $Z_1(X)$ and $Z_1(Y)$.
In particular, their minimal cycle bases are the same (up to rescaling).
To capture the structural difference between $X$ and $Y$, we introduce signed chains. 


\begin{figure}
    \begin{minipage}{.475\linewidth}
        \centering
        \includegraphics[width=\linewidth]{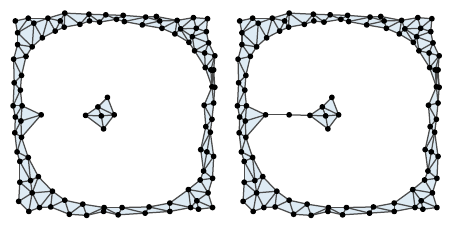}
        \caption{Unsigned cycles cannot distinguish $X$ (left) and $Y$ (right).}
        \label{fig:signed-vs-unsigned}
    \end{minipage}\hfill
    \begin{minipage}{.475\linewidth}
        \centering
        \includegraphics[width=\linewidth]{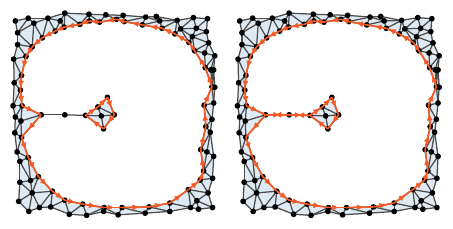}
        \captionsetup{textformat=simple}
        \caption{Difference between a unsigned (left) and signed (right) cycle in a simplicial complex.}
        \label{fig:signed-vs-unsigned-triangulation}
    \end{minipage}
\end{figure}

\begin{definition}
    For a simplicial complex $L$, the \emph{signed $k$-chains $C_k^\pm(L)$} 
    are the vector space spanned by symbols $\sigma^+$, $\sigma^-$ for each $k$-simplex $\sigma\in L$.
    Let $\iota$ be the involution defined by $\iota(\sigma^\pm) \coloneqq \sigma^\mp$,
    where $\sigma^\pm$ stands for either $\sigma^+$ or $\sigma^-$.\enlargethispage{\baselineskip}
    Define the map
    \[
        \begin{split}
            \partial_n^\pm\colon C_k^\pm(L) &\to C_{k-1}^\pm(L), \\
            [v_0,\dots,v_n]^\pm&\mapsto \sum_{i=0}^k \iota^i \big([v_0,\dots,v_{i-1}, v_{i+1},\dots, v_k]^\pm\big),
        \end{split}
    \]
    where $\iota^{\pm k}$ refers to the $k$-fold composition of $\iota$ with itself.
\end{definition}

We call the generators $\sigma^\pm$ of $C^\pm_k(L)$ \emph{signed simplices}.
Note that $(C_\bullet^\pm(L), \partial_\bullet^\pm)$ is not a chain complex. 
There is a natural map $\pi^\pm_k\colon  C_k^\pm(L)\to C_k(L)$ given by $\sigma^\pm \mapsto \pm\sigma$.
We have 
\[ 
    \partial_k \circ \pi^\pm_k = \pi^\pm_{k-1}\circ \partial_k^\pm.
\]
The \emph{signed $k$-cycles} of $L$ are $Z_k^\pm(L) \coloneqq \ker (\partial_k\circ\pi^\pm_k)$. If $z \in Z^\pm_k(L)$ is a signed cycle, we write $[z] \coloneqq [\pi^\pm_k(z)] \in H_k(L)$. For $z\in Z_d^\pm(L)$, we set $\Int(z)=\Int(\pi^\pm_d(z))$.
Consider the map
\[
    \begin{split}
        \Phi_L^\pm\colon C^0(\bar{L}) &\to C^\pm_d(T),\\
        \bar{\sigma}^* &\mapsto \iota^{\frac12(\omega(\sigma)-1)} (\sigma^+);
    \end{split}
\]
cf.\ the definition of $\Phi_L$ in \eqref{eq:def-Phi}.
Clearly, $\pi^\pm_d\Phi_L^\pm = \Phi_L$.
Consider the projection $p_L\colon C^\pm_d(T) \to C^\pm_d(L)$. 
We obtain the composition $\Theta_L^\pm\coloneqq p_L \partial^\pm\Phi_L^\pm\colon C^0(\bar{L}) \to C_d(L)$.
Observe that $\pi^\pm_d\Theta^\pm_L = \Theta_L$.
If $K$ is as in \cref{def:persistence-forest}, 
we define its \emph{signed persistence forest} $T^\pm(K)$ (w.r.t. an orientation $\omega$ of $K$)
as in \cref{def:persistence-forest}, with $\Theta^\pm$ instead of $\Theta$.
Replacing $\omega$ with $-\omega$ transforms $\Theta^\pm_{K_t}(-)$ into $\iota(\Theta^\pm_{K_t}(-))$.
Since the orientation on $K$ is unique up to sign, the signed persistence forest on $K$ is unique up to application of $\iota$.

Applying $\pi^\pm_d$ to the labels yields the persistence forest as defined in \cref{def:persistence-forest}.
Observe that only $0$ and $1$ can appear as coefficients in a signed chain of the form $\Theta^\pm_{K_t}(c)$ for $c\in \tilde{\pi}_0(\bar L)$.
Thus, the choice of coefficient field does not affect the (signed) persistence forest.


\subsection{The persistence forest algorithm}
We say that a signed chain \emph{contains} the signed simplex $\sigma^\pm$ if the coefficient of $\sigma^\pm$ is non-zero.
For an unsigned $d$-simplex $\sigma\in K$ and signed $d$-chains $z_1,z_2\in C_d^\pm(L)$,
we define the functions $\Merge_\sigma(z_1,z_2) = z_1 + z_2 - \sigma^*_+(z_1+z_2)\sigma^+ - \sigma^*_-(z_1+z_2)\sigma^-$
and $\Cancel_\sigma(z) = z - \sigma^*_+(z)\sigma^+ - \sigma^*_-(z)\sigma^-$,
where $\sigma^*_\pm$ denotes the dual of $\sigma^\pm$.

\begin{algorithm}[tb]
    \caption{Persistence Forest Algorithm}
    \label{alg:persistence-forest-alogrithm}
    \KwIn{A filtered simplicial complex $K$ as in \cref{thm:persistence-forest-algorithm} with filtration function $\tau$ and orientation $\omega$}
    \KwOut{A signed persistence forest $T^\pm(K)$ of $K$}
    choose an enumeration $K = \{\sigma_1,\dotsc,\sigma_{\abs{K}}\}$ of $K$ such that $\sigma_i\subseteq\sigma_j$ implies $i\leq j$\;
    let $A = \emptyset$, $V \gets \emptyset$ and $E \gets \emptyset$\;
    \For{$i = \abs{K},\dotsc,1$}{
        \If{$\dim \sigma_i = d+1$}{
            add a new vertex $v=(\iota^{\frac12(\omega(\sigma)-1)}\partial^\pm\sigma_i,\tau(\sigma_i))$ to $V$\;
            add $(\iota^{\frac12(\omega(\sigma)-1)}\partial^\pm\sigma_i, v)$ to $A$\;
        }
        \ElseIf{$\dim\sigma_i = d$}{
            find $(z^+, v^+)$ and $(z^-,v^-) \in A$ such that $z^\pm$ contains $\sigma^\pm$, if they exist\;
            \If{$(z^\pm, v^\pm)$ was found but not $(z^\mp, v^\mp)$}{
                add a new vertex $w = (z^\pm,\tau(\sigma_i))$ to $V$\;
                add a new edge $e\colon v \to w$ to $E$\;
                remove $(z^\pm,v^\pm)$ from $A$\;
            }
            \ElseIf{both were found and are unequal}{
                let $z = \Merge_{\sigma_i}(z^+,z^-)$\;
                add a new vertex $w= (z,\tau(\sigma_i))$ to $V$\;
                add two new edges $e^\pm\colon w \to v^\pm$ to $E$\;
                add $(z, w)$ to $A$ and remove both $(z^\pm, v^\pm)$ from $A$\;
            }
            \ElseIf{both were found and are equal}{
                let $z = \Cancel_{\sigma_i}(z^+)$\;
                add a new vertex $w=(z,\tau(\sigma_i))$ to $V$\;
                add a new edge $e\colon w \to v$ to $E$\;
                add $(z, w)$ to $A$ and remove $(z^+, v^+)$ from $A$\;
            }
        }
    }
    contract all edges $(z_1, t_1) \to (z_2, t_2)$ with $t_1 = t_2$ and remove vertices of degree $0$\;
    \Return $(V, E)$\;
\end{algorithm}

\begin{theorem}
    \label{thm:persistence-forest-algorithm}
    Let $K$ be a $(d+1)$-dimensional closed sublevel filtered simplicial complex,
    such that $K_\infty$ is contractible and extends to a triangulation of $S^{d+1}$.
    Then the Persistence Forest Algorithm~\ref{alg:persistence-forest-alogrithm} returns the signed persistence forest $T^\pm(K)$ of $K$.
\end{theorem}

The unsigned persistence forest $T(K)$ of $K$ is obtained by applying $\pi^\pm$ to $T^\pm(K)$.
Note that the persistence forest algorithm does not depend on the choice of coefficient field.

\begin{proof}[Proof of \cref{thm:persistence-forest-algorithm}]
    W.l.o.g., we may assume that $\tau\colon K_\infty \to \Z$ is injective,
    i.e., is a simplex-wise filtration.
    Enumerate the simplices of $K$ such that
    $K_i=\Set{\sigma_1,\dotsc,\sigma_{i}}$.
    By slight abuse of notation, we assume that the signed dual boundary $\Theta^\pm_{K_i}$ and the projection $p_{K_i}$ map into $C_d^\pm(K_\infty)$.
    We proceed to show that at the end of step $i$ of the iteration, there is a bijection between $A$ and $\tilde{\pi}_0(\overline{K_i})$; 
    i.e., for each $(z,v) \in A$, we have $z= \Theta^\pm_{K_i}(c)$ for some connected component $c\in \tilde{\pi}_0(\overline{K_i})$. 
    
    We first note that if $\dim \sigma_i=d+1$, 
    then removing $\sigma_i$ from $K_i$ creates the new connected component $\{\bar\sigma_i\}$ in $\overline{K_i}$ and $\Theta^\pm_{K_i}(\Set{\bar\sigma_i})=\iota^{\frac12(\omega(\sigma)-1)}\partial^\pm\sigma_i$.

    If $\dim \sigma_i = d$, it corresponds to an edge $\bar\sigma_i \in \overline{K_i} \setminus \overline{K_{i-1}}$.
    A $d$-simplex in $K$ has at most two cofaces, $\tau_1$ and $\tau_2$.
    The elements $(z^+, v^+),(z^-,v^-) \in A$ from the algorithm correspond to the connected components of these cofaces in $\overline{K_i}$.
    
    If neither $(z^+, v^+)$ nor $(z^-,v^-)$ exist, then the two vertices $\bar\tau_1$, $\bar\tau_2$ connected by the edge $\bar\sigma_i$ lie in the connected component of $\infty$.
    Therefore, inserting $\bar\sigma_i$ into $\overline{K_i}$ does not change the bounded connected components.
    If $(z^\pm, v^\pm)$ exists but not $(z^\mp, v^\mp)$, we know that $\bar\sigma_i$ connects a bounded connected component to $\infty$, 
    which corresponds to a root in the persistence forest.
    If both $(z^+, v^+)$ and $(z^-,v^-)$ are found and are distinct, we deduce that $\bar\sigma_i$ connects two bounded connected components $c^+,c^- \in\tilde{\pi}_0(\overline{K_i})$, 
    which corresponds to a merger of two tree branches in the persistence forest.
    The correspondence between connected components and signed cycles yields $\Theta^\pm_{K_i}(c^\pm)=z^\pm$.
    Since $\Theta^\pm_{K_i}(c^+) + \Theta^\pm_{K_i}(c^-)= \Theta^\pm_{K_i}(c^+\cup c^-)$, we deduce
    \[
        \Merge_{\sigma_i}(\Theta^\pm_{K_i}(c^+),\Theta^\pm_{K_i}(c^-))
        = p_{K_{i-1}}(\Theta^\pm_{K_i}(c^+ \cup c^-))=\Theta^\pm_{K_{i-1}}(c^+ \cup c^-).
    \]
    In the case $(z^+, v^+)=(z^-,v^-)$, the edge $\bar\sigma_i$ connects two points in the connected component $c \in\tilde{\pi}_0(\overline{K_i})$, and $z^+ = z^- = \Theta^\pm_{K_i}(c)$. 
    It is easy to verify that 
    \[
        \Cancel_{\sigma_i}(\Theta^\pm_{K_i}(c)) = p_{K_{i-1}}(\Theta^\pm_{K_i}(c))=\Theta^\pm_{K_{i-1}}(c).
    \]
    After the for-loop terminates, the graph $(V,E)$ is the signed persistence forest of $K$.
    For general $\tau$,
    we obtain the merge forest of $K$ by contracting all edges such that start and end point have the same filtration value, and deleting redundant vertices of degree $0$ afterwards.
    With contracting a parent-child pair in a forest, we refer to the process of deleting the child vertex and all its edges, and adding new edges from the all children of the child vertex to the parent vertex.
    If the reduction process leads to a root with multiple children, we create a copy of that root node for each child, add an edge from each child to its new root, and delete the original root node.
\end{proof}

\subparagraph*{Algorithmic Complexity} \label{rmk:complexity} Observe that it suffices to iterate over the simplices in $K_\infty^{(d)}\cup K_\infty^{(d+1)}$, and $\#K_\infty^{(d+1)}\leq \#K_\infty^{(d)}$. 
Let $L\coloneqq\max\Set{\#\text{supp}(z); z \text{ appears in } A}\leq \#K_\infty^{(d)}$ be the number of simplices in the largest appearing cycle.
For a $d$-simplex $\sigma$, we can find ${(z^+, v^+),(z^-,v^-) \in A}$ contaning $\sigma^+,\sigma^-$ by finding the roots $v_1,v_2$ above the leaves induced by its cofaces $\tau_1,\tau_2\in K_\infty^{(d+1)}$. 
Let $z_1,z_2$ be the signed cycles at $v_1,v_2$.
If $(z_1,v_1),(z_2,v_2)\in A$, then $\Set{(z_1,v_1),(z_2,v_2)}= \Set{(z^+, v^+),(z^-,v^-)}$, and otherwise either $(z^+, v^+)$ or $(z^-,v^-)$ does not exist.
If $\sigma$ only has a single coface in $K_\infty^{(d+1)}$, then $(z^+, v^+)$ or $(z^-,v^-)$ does not exist.
With proper bookkeeping, this procedure can be done in $\OO(L)$ since the number of nodes between a leaf and its root is bounded by $\OO(L)$.
Since we can also compute $\Merge_\sigma(z_1,z_2)$ in $\OO(\#\text{supp}(z_1)+\# \text{supp}(z_2))\leq \OO(L)$, the complexity of the $i$-th iteration is $\OO(L)$.
The overall algorithmic complexity is thus bounded by $\OO\big(L\cdot \#K_\infty^{(d)}\big)\leq\OO\big((\# K_\infty^{(d)})^2\big).$

For a point cloud in $\R^{d+1}$ with $n$ points, we can use persistence forest algorithm on a filtration of the Delaunay complex,
which has up to $\OO\big(n^{\left\lceil\frac{d+1}{2}\right\rceil}\big)$ simplicies \cite{Seidel:1995}, and can be computed in $\OO(n\log n + n^{\left\lceil\frac{d+1}{2}\right\rceil})$, see \cite{EdelsbrunnerShah:1996}. 
This yields an overall complexity of $\OO\big(n^{2\left\lceil\frac{d+1}{2}\right\rceil}\big)$. 
We note that the expected size of the Delaunay triangulation is $\OO(n)$ in many settings \cite{Dwyer:1991,GolinNa:2003}, and can be computed with expected complexity of $\OO(n\log n)$ for i.i.d. points \cite{EdelsbrunnerShah:1996}.

\begin{proposition}\label{eq:forest-cycle-compatatibility}
    If $(z_1,t_1),(z_2,t_2)$ are vertices in the (signed) persistence forest of $K$ such that there exists a directed path from $(z_1,t_1)$ to $(z_2,t_2)$, then $\supp z_1 \subseteq \left< \supp \Int z_2 \right>$.
\end{proposition}
\begin{proof}
    By construction of the persistence forest, we get $c_1\in\pi_0(\bar K_{t_1})$ and $c_2\in\pi_0(\bar K_{t_2})$ such that $z_1=\Theta_{K_{t_1}}(c_1), z_2=\Theta_{K_{t_2}}(c_2)$ and $c_1\subseteq c_2$. Using \cref{eq:interior-of-unsigned-dual-boundary}, we conclude
    \[\supp\Int \Theta_{K_{t_1}}(c_1) = \Set{\sigma\in K^{(d+1)}_\infty; \bar\sigma\in c_1}\subseteq\Set{\sigma\in K^{(d+1)}_\infty; \bar\sigma\in c_2}=\supp\Int \Theta_{K_{t_2}}(c_2).\]
    We thus know that $\supp z_1 \subseteq \left< \supp \Int z_1 \right> \subseteq \left< \supp \Int z_2 \right>$.
    The signed case follows immediately since $\Int \Theta^\pm_{K_{t_1}} = \Int \Theta_{K_{t_1}}$.
\end{proof}

\section{Generalized Landscapes}
\label{sec:landscapes}
\subsection{Cycle Progression Barcode}
Let $K$ be as in \cref{thm:persistence-forest-algorithm}.
Recall from \cref{thm:interior} that every $c \in Z_d(K_\infty)$ bounds a uniquely determined chain $\Int(c) \in C_{d+1}(K_\infty)$.
\begin{definition}
    \label{def:cpb}
    A \emph{cycle progression barcode} of $H_d(K)$ is a set $\Gamma$ containing a pair $(I,\gamma)$ with $\gamma\colon \R \to Z_d(K_\infty)$
    for each $I \in \barc(H_d(K))$, 
    such that
    \begin{enumerate}
        \item $\gamma(r)\in Z_d(K_r)$ for every $r\in \R$ and $(I,\gamma) \in \Gamma$,
        \item $\gamma(r)=0$ for all $(I,\gamma) \in \Gamma$ and $r\not\in I$,
        \item\label{def:cpb-3} the set $\Set{\gamma(r); \text{$(I,\gamma) \in \Gamma$ and $r \in I$}}$ 
            is a cycle basis of $K_r$,
        \item\label{def:cpb:compatibility1}
            for $r \leq s$, we have $\supp \gamma(s) \subseteq \langle\supp \Int\gamma(r)\rangle$, where the right side is the subcomplex of $K_\infty$ generated by the simplices on which $\Int\gamma(r)$ is supported
        \item\label{def:cpb:compatibility2} the $[\gamma(s)]$-coefficient of $[\gamma(r)]$
        w.r.t.\ the basis $\Set{[\gamma'(s)]; (I', \gamma') \in \Gamma, s \in I'}$ of $H_d(K_s)$ is $1$.
    \end{enumerate}
    We call a cycle progression barcode \emph{minimal} w.r.t. a weight $\wt\colon (T^*)^{(d+1)}\to\R_{> 0}$ if the cycle basis in \ref{def:cpb-3}
    is minimal w.r.t. $\wt$ for all $r$.
\end{definition}

\begin{example}
    \begin{figure}[b]
    	\foreach \i in {2,...,11}{
            \begin{tikzpicture}[every edge quotes/.append style={inner sep=0pt, font=\tiny}, x=.5cm, y=.5cm, baseline={(0,.3)}]
    			\fill[fill=gray!25, node distance=1] 
    			{ [nodes={inner sep=1pt, fill=black, shape=circle}]
    				node (1) {}
    				\ifnum\i>1 node[right=of 1] (2) {}\fi
    				\ifnum\i>2 node[above right=sin(60) and cos(60) of 1] (3) {}\fi
    				\ifnum\i>3 node[above right=sin(60) and cos(60) of 2] (4) {}\fi
    				\ifnum\i>5 node[above right=sin(60) and cos(60) of 3] (5) {}\fi
    			}
    			\ifnum\i>1 (1) edge["1"'] (2)\fi
    			\ifnum\i>2 (1) edge["2" ] (3)\fi
    			\ifnum\i>3 (2) edge["3"'] (4)\fi
    			\ifnum\i>4 (3) edge["4" ] (4)\fi
    			\ifnum\i>5 (3) edge["5" ] (5)\fi
    			\ifnum\i>6 (4) edge["6"'] (5)\fi
    			\ifnum\i>7 (2) edge["7" ] (3)\fi
    			{ [nodes={red, font=\tiny}]
    				\ifnum\i>8  (2.center)--(3.center)--(4.center)--cycle node at (barycentric cs:2=1,3=1,4=1) {8}\fi
    				\ifnum\i>9  (1.center)--(2.center)--(3.center)--cycle node at (barycentric cs:1=1,2=1,3=1) {9}\fi
    				\ifnum\i>10 (3.center)--(4.center)--(5.center)--cycle node at (barycentric cs:3=1,4=1,5=1) {10}\fi
    			};
    			\ifthenelse{\i>4 \and \i<8}{\draw[thick, green] (1) -- (2) -- (4) -- (3) -- (1);}{}
    			\ifthenelse{\i>7 \and \i<10}{\draw[thick, green] (1) -- (2) -- (3) -- (1);}{}
    			\ifthenelse{\i>6 \and \i<11}{\draw[thick, blue] (4) -- (5) -- (3) -- (4);}{}
    			\ifthenelse{\i>7 \and \i<9}{\draw[thick, red, dashed] (2) edge[solid] (4) (4) -- (3) -- (2);}{}
    		\end{tikzpicture}
            \ifnum\i<11 $\subseteq$\fi
    	}
        \caption{%
            Cycles belonging to a minimal cycle progression barcode of the filtered complex.
            The colors of the cycles correspond to the maximal paths in the persistence forest in \cref{fig:persistence-tree-and-forest}.
            }
        \label{fig:cycle-progression}
    \end{figure}
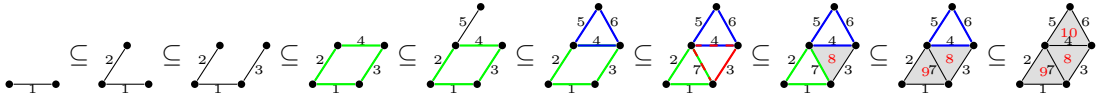
    Recall the complex from \cref{fig:persistence-tree-and-forest}.
    Ignoring sign, a minimal cycle progression barcode of it is given by the cycles in \cref{fig:cycle-progression}.
\end{example}

\begin{algorithm}[bp]
    \caption{Obtaining a cycle progression barcode from a persistence forest.}
    \label{alg:forest-to-barcode}
    \KwIn{A persistence forest $(V,E) = T(K)$ of $K$}
    \KwOut{A minimal cycle progression barcode $\Gamma$ of $H_d(K)$}
    $E \gets \text{edges of $T(K)$}$\;
    $\Gamma \gets \emptyset$\;
    \While{$E \neq \emptyset$}{
        $p \gets \text{a directed path $(x,d) \leadsto (y,b)$ in $E$ of maximal length}$\;
        $I \gets [b,d)$\;
        $\gamma \gets (r \mapsto \text{the $v$ for $(v,t)$ on $p$ with minimal $t > r$})$\;
        $\Gamma \gets \Gamma \cup \{(I, \gamma)\}$\;
        $E \gets E \setminus \{p\}$\;
    }
    \Return $\Gamma$\;
\end{algorithm}
\begin{theorem}
    \label{thm:forest-to-barcode}
    Let $K$ be as above, and $T(K)$ be its (signed) persistence forest.
    Then \cref{alg:forest-to-barcode} computes a cycle progression barcode of $H_d(K)$ which is minimal for any weight.
\end{theorem}
\begin{proof}
    From \cref{cor:filtration-alexander-duality}, we deduce that the well-known elder-rule determines the intervals of $\barc(H_d(K))$.
    The representatives $\gamma(r)$ give a cycle basis for each $r$ by the definition of $T(K)$.
    Minimality follows from minimality of the representatives in $T(K)$ by \cref{thm:weigth-optimality}.
\end{proof}
Signed cycle progression barcodes are defined analogously, and can also be obtained from signed persistence forests using \cref{alg:forest-to-barcode}.
We call a signed cycle progression barcode $\Gamma^\pm$ \emph{minimal} if $\Set{(I, \pi^\pm(\gamma)); (I,\gamma) \in \Gamma^\pm}$
is a minimal cycle progression barcode and $\gamma(r) = p_{K_r}(\partial^\pm(\Int(\gamma(r))))$ for all $(I, \gamma) \in \Gamma^\pm$.
For a signed persistence forest, \cref{alg:forest-to-barcode} yields a minimal signed cycle progression barcode.

\begin{proposition} \label{prop:prog-barcode-uniqueness}
    Let $K$ be as in \cref{thm:persistence-forest-algorithm}
    and let $\wt\colon (T^*)^{(d+1)}\to\R_{> 0}$ be a weight function. If all intervals in $\barc(H_d(K))$ have pairwise distinct upper endpoints, then the minimal (signed) cycle progression barcode for $H_d(K)$ w.r.t. $\wt$ is unique up to rescaling.
\end{proposition}
\begin{proof}
    Let $\Gamma$ be a minimal cycle progression barcode. Let $(I,\gamma) \in \Gamma$ and $s,t\in I$ with $s\leq t$. 
    By \cref{thm:weigth-optimality}, weight optimality implies that there exist $c_s \in \tilde{\pi}_0(\overline{K_s})$ and $c_t \in \tilde{\pi}_0(\overline{K_t})$
    such that $\gamma(s)=a_s\Theta_{K_s}(c_s)$ and $\gamma(t)=a_t\Theta_{K_t}(c_t)$ for $a_s,a_t \in \kk$.
    With \cref{eq:dual-boundary-compatibility} and \cref{def:cpb}.\ref{def:cpb:compatibility2},
    we get $a_s=a_t$ and $c_s\supseteq c_t$. 
    We conclude that $(I,\gamma)$ is induced by a path in the persistence forest. 
    \Cref{def:cpb}.\ref{def:cpb-3} ensures that $\gamma(t)\ne\gamma'(t)$ for distinct $(I,\gamma)$,$(J,\gamma')\in\Gamma$ and $t\in I\cap J$, 
    from which we deduce that the induced paths can at most intersect in one of their endpoints. 
    As the death times in $\barc(H_d(K))$ are distinct, the set of suitable paths in the persistence forest is uniquely determined by the elder rule. The uniqueness in the signed setting is a direct consequence of the unsigned case.
\end{proof}

\subsection{Functionals on Cycle Representatives Barcode}
By a \emph{functional}, we mean a map $f\colon Z_d(K_\infty) \to \R$,
and analogously for signed chains.

\begin{example}
    \label{ex:functionals}
    Let $K \subseteq \R^{d+1}$ be as in \cref{thm:persistence-forest-algorithm}.
    \begin{itemize}
        \item The $k$-dimensional volume  (i.e., for $k=1$, the arc length of $c$) for a (signed) chain $c\in C_k^{(\pm)}(K)$ is given by $\Vol(c) \coloneqq  \sum_{\sigma\in\supp c}\Vol(\sigma)$, where $ \Vol([v_0,\dotsc,v_k]^{(\pm)}) = \frac{1}{d!}\abs{\det(v_1-v_0,\dotsc,v_k-v_0)}$.
        \item The enclosed $d+1$-dimensional volume (i.e., for $d=1$, the enclosed area) for a (signed) cycle $z \in Z_d^{(\pm)}(K)$ is defined as  $\EVol(z) \coloneqq \Vol(\Int z)$.
        \item For $d = 1$ and coefficients in $\Z$, a $1$-cycle $z$ can be written as a sum $z = \sum_{u \in U} u$ of closed polygonal chains $u = \sum_{i=1}^n [v_i, v_{i+1}]$, taking indices mod $n$.
            For such a chain $u$, we define the \emph{excess curvature} $\kappa(u) = \frac{\varkappa(u)}{2\pi}-1$,
            where $\varkappa(u)$ is the \emph{total curvature} $\varkappa(u) = \sum_i \abs{\angle(v_i-v_{i-1},v_{i+1}-v_i)}$,
            and for $c$, we define $\kappa(z) = \sum_{u \in U} \kappa(u)$. Note that $\kappa(z) = 0$ if and only if every polygonal chain in $z$ is convex.

            If the decomposition $z = \sum_{u \in U} u$ of $z$ is not unique (which happens if $c$ contains edges twice, or if $z$ contains more than two edges incident to the same vertex),
            we chose the decomposition that minimizes $|U|$, and in case of equality also minimizes $\sum_{u \in U} \kappa(u)$.

            Since a signed $1$-chain induces a multiset of oriented edges, we can construct a polyhedral chain and define excess curvature analogously to the approach above. 
    \end{itemize}
    Observe that $\Vol(z^\pm)\geq\Vol(\pi^\pm(z^\pm))$ and  $\EVol(z^\pm)=\EVol(\pi^\pm(z^\pm))$.
\end{example}

Let $\Gamma$ be a cycle progression barcode for $K$, and $(I,\gamma) \in \Gamma$.
For a functional $f$, we consider the composition $f\circ\gamma$.
Additionally, we also consider the rescaled convolution
$f \boxast \gamma \coloneqq \frac{1}{2} ((f\circ\gamma)\ast\one_I)(2x-\frac{b+d}{2})$,
where  $\one_I$ denotes the indicator function on the interval $I = [a,b)$, and $g*h$ denotes the usual convolution $(g*h)(t) = \int_{-\infty}^\infty g(\tau) h(t-\tau)\mathrm{d}\tau$.
The rescaling in the definition of $f \boxast \gamma$ ensures the following:

\begin{example}
    If $f = \one\colon t \mapsto 1$ is the constant function, then we get $f\circ \gamma = \one_I$,
    and $(f \boxast \gamma)(t) = \min\{0, t-b, d-t\}$, where $I = [b,d)$. This is precisely the pyramid function supported on $I$ with slope $\pm 1$ which is used for standard persistence landscapes.
\end{example}

\noindent For a totally ordered set $S$, let $\nthmax{n} S$ denote the $n$-th largest element of $S$.

\begin{definition}
    Let $K$ be as in \cref{prop:prog-barcode-uniqueness} and $\Gamma$ be a minimal cycle progression barcode of $K$. The \emph{$n$-th generalized persistence landscape} of $K$ with respect to the functional $f$ and the cycle progression barcode $\Gamma$
    is the function 
    \begin{align*}
        \lambda^\Gamma_n(K, f)\colon \R &\to \R,
        \\ t &\mapsto \nthmax{n} \Set{(f \boxast \gamma)(t); (\gamma, I) \in \Gamma}.
    \end{align*}
\end{definition}

The \emph{$n$-th generalized signed persistence landscape} $\lambda^{\Gamma^\pm}_n(K, f)$ is defined analogously for a minimal cycle progression barcode $\Gamma^\pm$. 
We note that if all intervals in $\barc(H_d(K))$ have pairwise distinct upper endpoints and the functional $f$ satisfies $f(az)=f(z)$ for all $a\in\kk\setminus\Set{0}$, then the generalized (signed) landscape is independent of the choice of minimal progression barcode by \cref{prop:prog-barcode-uniqueness}. The condition $f(az)=f(z)$ for $a\in\kk\setminus\Set{0}$ holds for both the $d$-dimensional volume and the $d+1$-dimensional enclosed volume defined in $\cref{ex:functionals}$. The excess curvature satisfies $\kappa(z)=\kappa(-z)$, which ensures that the generalized landscape w.r.t. $K$ and the cycle progression barcode given by \cref{alg:persistence-forest-alogrithm,alg:forest-to-barcode} does not depend on the choice of orientation.
Unless specified otherwise, we use the minimal (signed) cycle progression barcode $\Gamma^{(\pm)}$ from \cref{alg:persistence-forest-alogrithm,alg:forest-to-barcode}, and simply write $\lambda_n(K, -)$. 

\begin{example}
    For the constant function $\one$, we obtain that $\lambda_n(K, \one)$
    is the $n$-th persistence landscape in the usual sense.
\end{example}

\section{Examples and Benchmarks}
\label{sec:examples-and-benchmarks}

\subparagraph*{Generalized Landscapes} We demonstrate the effectiveness of generalized persistence landscapes as a shape descriptor by computing them for some exemplary point clouds in $\R^2$ in \cref{fig:point-cloud-comparison}.
We sampled 1000 points each from a circle, an ellipse, a circle with “spoke”, and a five-leaved star.
Each sample is subject to normally distributed noise.
Using the $\alpha$-filtration for each point cloud, 
we compute the uniquely determined minimal cycle progression barcode $\Gamma$ by combining \cref{alg:persistence-forest-alogrithm} and \ref{alg:forest-to-barcode}.
For the pair $(I,\gamma)\in \Gamma$ corresponding to the longest interval $I$ in the $H_1$-barcode, 
we plot $f\circ\gamma$ for every $f\in \Set{\text{length, area, excess curvature}}$. Additionally, we compare the generalized landscapes $\lambda_1(-, f)$ for these point clouds.
All four samples yield similar barcodes but are clearly distinguishable from $f\circ\gamma$ and $\lambda_1(-, f)$.

\begin{figure}
    \centering
    \includegraphics{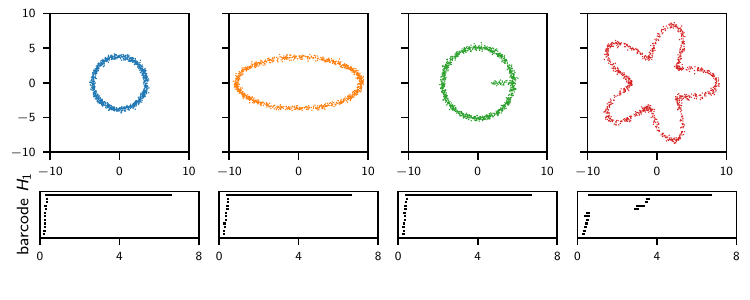}
    \includegraphics{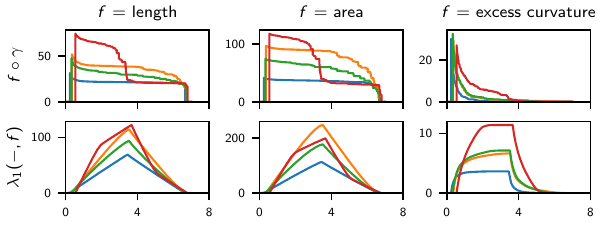}
    \caption{%
        Four different point clouds and the longest 12 bars in their $H_1$-barcode.
        Bottom: for $f$ denoting length, area, and excess curvature,
        we plot $f \circ \gamma$ for the cycle progression $\gamma$ corresponding to the longest bar, 
        and the first landscape $\lambda_1(-, f)$ for each point cloud. The colors correspond to the point clouds.
    }
    \label{fig:point-cloud-comparison}
\end{figure}

\subparagraph*{Signed and Unsigned Landscapes} To illustrate the utility of signed chains,
we show in \cref{fig:signed-vs-unsigned-connected-components} that they can detect bridges between connected components more accurately than unsigned chains. 
We compare the number of excess connected components (i.e., $f(c) = \abs{\pi_0(c)}-1$) between point clouds formed by a circle with three interior point clusters, both with and without bridges connecting them to the circle.
The induced barcode functional $f\circ \gamma$ corresponding to the longest bar and the generalized landscapes for signed chains clearly detect the bridges whereas their unsigned counterpart incorrectly counts up to three excess connected components even if the bridges are present.

\begin{figure}[tb]
    \centering
    \raisebox{-.4\height}{\includegraphics{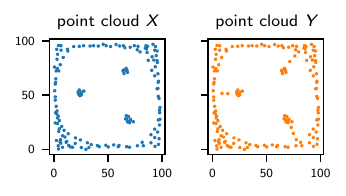}}\hfill
    \raisebox{-.4\height}{\includegraphics{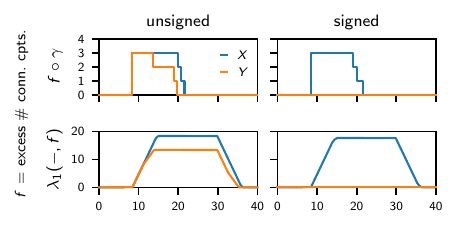}}
    \caption{%
        Signed chains capture bridges between connected components more accurately.
        Left: The two point clouds considered.
        Right: Number of excess connected components in the longest bar (upper row) and the associated persistence landscape (lower row) for unsigned and signed cycle representatives,
        for both point clouds.
    }
    \label{fig:signed-vs-unsigned-connected-components}
\end{figure}

\subparagraph*{Computational Efficiency}
A Python implementation of \cref{alg:persistence-forest-alogrithm,alg:forest-to-barcode}, and of generalized landscapes
is available at \url{https://anonymous.4open.science/r/LoopForest}.
For a point cloud in $\R^{d+1}$, we compute its $\alpha$-filtration using Gudhi 3.11.0 \cite{gudhi:AlphaComplex, gudhi:FilteredComplexes}.
All benchmarks were performed using python 3.13.3 on a MacBook Pro (14-inch, 2024) equipped with an Apple M4 Pro SoC (8 performance cores, 4 efficiency cores), 48 GB unified memory, running macOS Sequoia 15.6.1.

\begin{figure}[tb]
    \centering
    \includegraphics{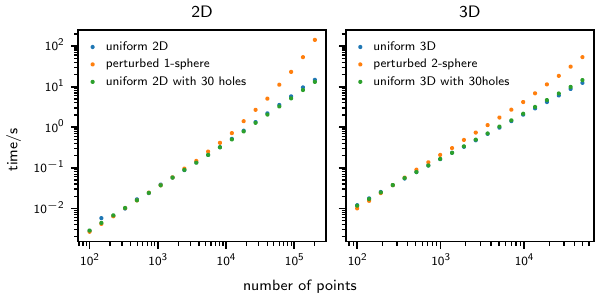}
    \caption{Runtime of the Persistence Forest Algorithm for point clouds in $\R^2$ and $\R^3$.}
    \label{fig:benchmark}
\end{figure}

To estimate the practical runtime of our pipeline  for $d \in \{1,2\}$, we used
uniformly distributed points in $[0,1]^{d+1}$,
uniformly distributed points a unit sphere $S^d$ perturbed by gaussian noise with $\sigma=0.05$, 
and uniformly distributed points in $[0,1]^{d+1}\setminus\bigcup_{i=1}^{30}B_i$, 
where each $B_i$ is a ball with random center point in $[0,1]^{d+1}$ and uniformly distributed $r\in[0,0.05]$.
We applied both \cref{alg:persistence-forest-alogrithm,alg:forest-to-barcode} to the $\alpha$-complexes of these point clouds,
but note that the runtime contribution of \cref{alg:forest-to-barcode} is negligible.
We plot the average runtime across 10 runs in \cref{fig:benchmark}.
The results show that we can compute volume-optimal cycle representatives across all filtration values for $10^4 - 10^5$ points in $\R^2$ and $\R^3$ within seconds.
This demonstrates that our approach is computationally feasible and applicable in practice.

\printbibliography
\end{document}